\documentclass[10pt, reqno]{amsart}
\usepackage{amsmath, amsthm, amscd, amsfonts, amssymb, graphicx, color}
\usepackage[bookmarksnumbered, colorlinks, plainpages]{hyperref}
\hypersetup{colorlinks=true,linkcolor=red, anchorcolor=green, citecolor=cyan, urlcolor=red, filecolor=magenta, pdftoolbar=true}
\usepackage{mathrsfs}

\newtheorem{theorem}{Theorem}[section]
\newtheorem{lemma}[theorem]{Lemma}

\newtheorem{corollary}[theorem]{Corollary}
\theoremstyle{definition}

\newtheorem{example}[theorem]{Example}

\theoremstyle{remark}
\newtheorem{remark}[theorem]{Remark}
\numberwithin{equation}{section}

\begin{document}
\setcounter{page}{1}

\title[K-theory of Jones polynomials]{K-theory of Jones polynomials}

\author[Glubokov]{Andrey ~Yu. ~Glubokov$^1$}

\author[Nikolaev]{Igor ~V. ~Nikolaev $^2$}

\date{\today}

\address{$^{1}$ Department of Mathematics,  Purdue University, 
150 N. University Street, West Lafayette, IN 47907-2067,  United States.}
\email{\textcolor[rgb]{0.00,0.00,0.84}{agluboko@purdue.edu}}

\address{$^{2}$ Department of Mathematics and Computer Science, St.~John's University, 8000 Utopia Parkway,  
New York,  NY 11439, United States.}
\email{\textcolor[rgb]{0.00,0.00,0.84}{igor.v.nikolaev@gmail.com}}

\dedicatory{In memory of V.~F.~R.~Jones}



\date{}

\begin{abstract}
We recover  the Jones polynomials of knots and links from 
the K-theory  of a cluster $C^*$-algebra of the sphere with two cusps. 
In particular, an interplay between the Chebyshev and Jones polynomials is studied. 
\end{abstract}

\maketitle

\section{Introduction}
Cluster algebras  are a   class of commutative rings  introduced by  [Fomin \& Zelevinsky 2002]  \cite{FoZe1}.
The  cluster algebra  of rank $n$ 
is a subring  $\mathcal{A}(\mathbf{x}, B)$  of the field  of  rational functions in $n$ variables
depending  on  variables  $\mathbf{x}=(x_1,\dots, x_n)$
and a skew-symmetric matrix  $B=(b_{ij})\in M_n(\mathbf{Z})$.
The pair  $(\mathbf{x}, B)$ is called a  seed.
A new cluster $\mathbf{x}'=(x_1,\dots,x_k',\dots,  x_n)$ and a new
skew-symmetric matrix $B'=(b_{ij}')$ is obtained from 
$(\mathbf{x}, B)$ by the   exchange relations [Williams 2014]  \cite[Definition 2.22]{Wil1}:
\begin{eqnarray}\label{eq1.1}
x_kx_k'  &=& \prod_{i=1}^n  x_i^{\max(b_{ik}, 0)} + \prod_{i=1}^n  x_i^{\max(-b_{ik}, 0)},\cr 
b_{ij}' &=& 
\begin{cases}
-b_{ij}  & \mbox{if}   ~i=k  ~\mbox{or}  ~j=k\cr
b_{ij}+{|b_{ik}|b_{kj}+b_{ik}|b_{kj}|\over 2}  & \mbox{otherwise.}
\end{cases}
\end{eqnarray}
The seed $(\mathbf{x}', B')$ is said to be a  mutation of $(\mathbf{x}, B)$ in direction $k$.
where $1\le k\le n$.  The  algebra  $\mathcal{A}(\mathbf{x}, B)$ is  generated by the 
cluster  variables $\{x_i\}_{i=1}^{\infty}$
obtained from the initial seed $(\mathbf{x}, B)$ by the iteration of mutations  in all possible
directions $k$.  

 The  Laurent phenomenon
proved by  [Fomin \& Zelevinsky 2002]  \cite{FoZe1}  says  that  $\mathcal{A}(\mathbf{x}, B)\subset \mathbf{Z}[\mathbf{x}^{\pm 1}]$,
where  $\mathbf{Z}[\mathbf{x}^{\pm 1}]$ is the ring of  the Laurent polynomials in  variables $\mathbf{x}=(x_1,\dots,x_n)$.
In particular, each  generator $x_i$  of  the algebra $\mathcal{A}(\mathbf{x}, B)$  can be 
written as a  Laurent polynomial in $n$ variables with the   integer coefficients. 
 The cluster algebra  $\mathcal{A}(\mathbf{x}, B)$  has the structure of an additive abelian
semigroup consisting of the Laurent polynomials with positive coefficients. 
In other words,  the $\mathcal{A}(\mathbf{x}, B)$ is a dimension group  \cite[Definition 3.5.2]{Nik1}.
We define the cluster $C^*$-algebra  $\mathbb{A}(\mathbf{x}, B)$  as   an  AF-algebra,  
such that $K_0(\mathbb{A}(\mathbf{x}, B))\cong  \mathcal{A}(\mathbf{x}, B)$
\cite[Section 4.4]{Nik1}.

 Denote by $S_{g,n}$  the Riemann surface   of genus $g\ge 0$  with the $n\ge 0$ cusps.
 Let   $\mathcal{A}(\mathbf{x},  S_{g,n})$ be the cluster algebra 
 coming from  a triangulation of the surface $S_{g,n}$  [Fomin,  Shapiro  \& Thurston  2008]  \cite{FoShaThu1} and 
  let $\mathbb{A}(\mathbf{x}, S_{g,n})$ be the corresponding cluster $C^*$-algebra. 
In what follows, we focus on the special case $g=0$ and $n=2$, i.e. when the surface  $S_{0,2}$ is 
 a sphere with two cusps.  The  $S_{0,2}$  is homotopy equivalent to an annulus
$\{z=u+iv\in\mathbf{C} ~|~ r\le |z|\le R\}$.   The AF-algebra  $\mathbb{A}(\mathbf{x}, S_{0,2})$
has the Bratteli diagram shown in Figure 1 and the 
 surface $S_{0,2}$  has an ideal triangulation with 
one marked point on each boundary component  [Fomin,  Shapiro  \& Thurston  2008, Example 4.4]  \cite{FoShaThu1} 
given by the matrix:
\begin{equation}\label{eq1.3}
B=\left(
\begin{matrix} 0 & 2\cr -2 & 0
\end{matrix}\right).
\end{equation}

\begin{figure}
\begin{picture}(100,100)(0,150)

\put(50,200){\circle{3}}

\put(33,188){\circle{3}}
\put(67,188){\circle{3}}

\put(50,177){\circle{3}}
\put(16,177){\circle{3}}
\put(84,177){\circle{3}}

\put(-1,164){\circle{3}}
\put(34,164){\circle{3}}
\put(68,164){\circle{3}}
\put(103,164){\circle{3}}


\put(49,199){\vector(-3,-2){15}}
\put(51,199){\vector(3,-2){15}}

\put(32,187){\vector(-3,-2){15}}
\put(34,187){\vector(3,-2){15}}

\put(66,187){\vector(-3,-2){15}}
\put(68,187){\vector(3,-2){15}}


\put(14,175){\vector(-3,-2){15}}
\put(17,175){\vector(3,-2){15}}

\put(49,175){\vector(-3,-2){15}}
\put(51,175){\vector(3,-2){15}}

\put(83,175){\vector(-3,-2){15}}
\put(86,175){\vector(3,-2){15}}

\put(-10,155){$\dots$}
\put(27,155){$\dots$}
\put(64,155){$\dots$}
\put(101,155){$\dots$}

\end{picture}
\caption{Pascal's triangle  diagram of the AF-algebra $\mathbb{A}(\mathbf{x}, S_{0,2})$.} 
\end{figure}
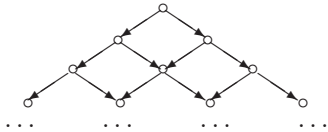

\bigskip
The aim of our note is
 an interplay between the Laurent phenomenon 
in  the cluster algebra $\mathcal{A}(\mathbf{x}, S_{0,2})\cong K_0(\mathbb{A}(\mathbf{x}, S_{0,2}))$
and the Jones  polynomials $V_L(t)$  [Jones 1985] \cite{Jon1}. 
The existence of such a link follows from a representation $\rho$ of the braid group $\mathscr{B}_k$
given by the formulas in \cite[Theorem 4.4.1]{Nik1}:
\begin{equation}\label{eq1.2}
\begin{cases} \mathscr{B}_{2g+1}\to \mathbb{A}(\mathbf{x}, S_{g,1}) &\cr
\mathscr{B}_{2g+2}\to \mathbb{A}(\mathbf{x}, S_{g,2}). &\end{cases}
\end{equation}

To formalize our results, denote by  $\mathscr{J}\subset\mathbf{Z} [t^{\pm {1\over 2}}]$
a ring generated by the set $\{ V_L(t) ~|~ L ~\hbox{runs all links}\}$. 
Denote by $x_1$ and $x_2$ the independent variables of 
 the rank 2 cluster algebra $\mathcal{A}(\mathbf{x}, S_{0,2})$. 
Let $\mathbb{P}\subset  \mathbb{A}(\mathbf{x}, S_{0,2})$  be 
an  AF-algebra defined by the truncated Pascal's  diagram shown in Figure 2
[Jones 1991] \cite[pp. 36, 50]{J1}. 
Our main results are as follows. 
\begin{theorem}\label{thm1.1}
There exists an inclusion of the rings  $\mathscr{J}\subset \mathcal{A}(\mathbf{x}, S_{0,2})$
and an isomorphism of the dimension groups $\mathscr{J}\cong K_0(\mathbb{P})$
induced by the  substitution: 
\begin{equation}
t^2= {2(x_1^2+x_2^2+1)\over x_1x_2}. 
\end{equation}
\end{theorem}
\begin{remark}
A combinatorial approach to  the Laurent phenomenon 
and the Jones  polynomials 
which is  based on the continued fractions and the snake graphs was
studied recently  by [Lee \& Schiffler 2019] \cite{LeeSch1}.
\end{remark}

\bigskip
The paper is organized as follows. Section 2 contains notation and definitions
necessary for the proof of theorem \ref{thm1.1}. 
The proof of theorem \ref{thm1.1} is given in Section 3. 
In Section 4 we calculate the Jones polynomials of two unlinked unknots, 
the Hopf link and the trefoil knot using theorem \ref{thm1.1}. 
A discussion of related results can be found in Section 5.

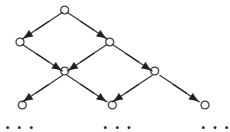
\begin{figure}
\begin{picture}(100,100)(0,150)

\put(50,200){\circle{3}}

\put(33,188){\circle{3}}
\put(67,188){\circle{3}}

\put(50,177){\circle{3}}

\put(84,177){\circle{3}}

\put(34,164){\circle{3}}
\put(68,164){\circle{3}}
\put(103,164){\circle{3}}


\put(49,199){\vector(-3,-2){15}}
\put(51,199){\vector(3,-2){15}}


\put(34,187){\vector(3,-2){15}}

\put(66,187){\vector(-3,-2){15}}
\put(68,187){\vector(3,-2){15}}


\put(49,175){\vector(-3,-2){15}}
\put(51,175){\vector(3,-2){15}}

\put(83,175){\vector(-3,-2){15}}
\put(86,175){\vector(3,-2){15}}

\put(27,155){$\dots$}
\put(64,155){$\dots$}
\put(101,155){$\dots$}

\end{picture}
\caption{Truncated Pascal's triangle diagram of the AF-algebra $\mathbb{P}$.} 
\end{figure}

\section{Preliminaries}
We  shall briefly review the Jones polynomials,  skein relation and cluster algebras of rank 2. 
We refer the reader to  [Jones 1985]  \cite{Jon1},  [Jones 1991] \cite{J1}
and [Sherman \& Zelevinsky 2004] \cite{SheZe1}
for a detailed account.

\subsection{Jones polynomials}
By $A_n$ we denote  an $n$-dimensional von Neumann algebra
generated by the identity and the Jones projections $e_1,e_2,\dots, e_n$,  see [Jones 1985]  \cite{Jon1}
for the definition of $e_i$.   Such  projections
are known to satisfy the  relations
$e_ie_{i\pm 1}e_i  = {1\over   [\mathcal{M} : \mathcal{N}]} ~e_i,
~e_ie_j = e_je_i$, if $|i-j|\ge 2$,
and the trace formula:
$tr~(e_n x) =     {1\over   [\mathcal{M} : \mathcal{N}]} ~tr~(x), \quad\forall x\in A_n$.
The reader can verify that  the relations for the Jones projections $e_i$ coincide with
such for the  generators $\sigma_i$ of the braid group  after  an  adjustment 
of the notation
$\sigma_i \mapsto  \sqrt{t} [(t+1)e_i-1], 
~\left[  \mathcal{M} : \mathcal{N} \right] = 2+t+{1\over t}$.
One gets a family $\rho_t$ of  representations of the braid group $\mathscr{B}_n$ into 
the Jones algebra $A_n$.   To get  a topological invariant
of the closed braid $\hat b$ of $b\in \mathscr{B}_n$ coming from the 
trace (a character) of the representation,  one needs to choose 
a representation whose trace is invariant under the first and the second 
Markov moves of  the braid $b$.  
The trace is invariant  of  the first Markov move,
because  two similar matrices have the same trace
for any representation from the family  $\rho_t$.    For the second Markov move,
we  have the trace formula which (after obvious substitutions) takes the
form $tr~(b\sigma_n) = -{1\over t+1} tr~(b)
~tr~(b\sigma_n^{-1}) = -{t\over t+1} tr~(b)$.
In general $tr~(b\sigma_n^{\pm 1})\ne tr~(b)$,  but one can always
re-scale the trace to get the equality.  Indeed,  the second Markov move
takes the braid from $\mathscr{B}_i$ and replaces it by a braid from $B_{i-1}$;  
there is  a finite number of such replacements because  the algorithm stops
for $\mathscr{B}_1$.  Therefore  a finite number of re-scalings by the constants    
 $-{1\over t+1}$ and     $-{t\over t+1}$ will give a quantity invariant under
 the second Markov move;  the quantity is known as the Jones polynomial
 of the closed braid $\hat b$.
Let $b\in \mathscr{B}_n$ be a braid and $\exp(b)$ be the sum of all  powers of generators $\sigma_i$
and $\sigma_i^{-1}$  in the word presentation of $b$ and  let $L:=\hat b$ be the closure of $b$.
Thus an isotopy  invariant of the link $L$ is given by  the quantity: 
\begin{equation}\label{eq2.5}
V_{L}(t):=\left(-{t+1\over\sqrt{t}}\right)^{n-1}(\sqrt{t})^{\exp(b)}~tr~(b). 
\end{equation}

\subsection{Skein relation}
If  the links  differ from each other only in a small region, 
 the trace invariant (\ref{eq2.5}) can be calculated recursively.
 Namely, it is known that $V_K(t)=1$, 
where  $K$ is the unknot. Recall that any link $L$ can be obtained from $K$ by a finite number
of local operations of adding an overpass or underpass to the diagram of the link $L$.
Denote by  $L^+$ ($L^-$, resp.)  a link obtained by adding an overpass
(underpass, resp.) to the link $L$.
\begin{theorem}
{\bf (\cite[Theorem 12]{Jon1})}
Each $V_L(t)$ can 
 be obtained  from the surgery of  $K$ using the skein relation:
\begin{equation}\label{eq2.6}
{1\over t}V_{L^-} -tV_{L^+}=\left(\sqrt{t}-{1\over\sqrt{t}}\right)V_L.
\end{equation} 
\end{theorem}

\subsection{Cluster algebras of rank 2}
For a pair of positive integers $b$ and $c$,  we define the matrix:
\begin{equation}
B=\left(
\begin{matrix} 0 & b\cr -c & 0
\end{matrix}\right).
\end{equation}

\medskip
The reader can verify, that the exchange relations (\ref{eq1.1}) 
for $B$  take the form:
\begin{equation}\label{eq2.7}
x_{i-1} x_{i+1} =
\left\{
\begin{array}[h]{ll}
1+x_i^b &  \mbox{if} \quad \mbox{$i$ odd,} \\
1+x_i^c  & \mbox{if}  \quad \mbox{$i$ even.}
\end{array}
\right.
\end{equation}

\medskip
Let us consider  the field of rational functions in two
commuting independent variables $x_1$ and $x_2$ with the rational
coefficients. 
We shall write $\mathcal{A}(b,c)$ for  a cluster algebra of  rank 2 
 generated by the variables $x_i$
[Sherman \& Zelevinsky 2004]   \cite[Section 2]{SheZe1}. 
Denote by $\mathscr{B}$ a basis of the algebra $\mathcal{A}(b,c)$.
\begin{theorem}\label{thm2.1}
{\bf  (\cite[Theorem 2.8]{SheZe1})}
Suppose that  $b=c=2$ or $b=1$ and  $c=4$. Then  
$\mathscr{B}=\{x_i^px_{i+1}^q ~|~ p,q\ge 0\} ~\bigcup ~\{T_n(x_1x_4-x_2x_3) ~|~ n\ge 1\}$,
where $T_n(x)$ are the Chebyshev polynomials of the first kind.  
\end{theorem}

\section{Proof of theorem \ref{thm1.1}}
Let us outline the main idea. 
Recall that each cluster variable $x_i$ in $\mathcal{A}(\mathbf{x}, S_{0,2})$ 
can be obtained from the initial cluster $(x_1,x_2)$ by a sequence 
of mutations described by the exchange relations (\ref{eq1.1}).
Likewise,   each Jones polynomial $V_L(t)$ can be obtained from the
$V_K(t)$ of the unknot $K$  by a local surgery of the knot diagram described by  
the skein relation (\ref{eq2.6}).
Roughly speaking, it will be shown that (\ref{eq1.1}) and    (\ref{eq2.6})  are equivalent relations
modulo  the  substitution   $t^2= {2(x_1^2+x_2^2+1)\over x_1x_2}$. 
We spilt the proof in a series of lemmas.

\begin{lemma}\label{lm3.1}
The substitution
\begin{equation}\label{eq3.1} 
\begin{cases} W_{L^+}=\left(-{t+1\over\sqrt{t}}\right) ~V_{L^+}&\cr
W_{L^-}=\left(-{t+1\over\sqrt{t}}\right) ~V_{L^-} &\cr
W_L=(t^2-1)V_L &
\end{cases}
\end{equation}
brings the skein relation  (\ref{eq2.6}) to the form: 
\begin{equation}\label{eq3.2} 
W_L=t^2  ~W_{L^+}-W_{L^-}. 
\end{equation}
\end{lemma}
\begin{proof}
Indeed,  let us multiply both sides of the equation (\ref{eq2.6}) by  
${(t^2-1)\sqrt{t}\over t-1}$.  After obvious algebraic operations, the 
skein relation (\ref{eq2.6})  can be written in the form: 
\begin{equation}\label{eq3.3} 
(t^2-1)V_L=    {t+1\over \sqrt{t}}\left[ 
V_{L^-} -t^2V_{L^+}\right]. 
\end{equation}

\bigskip
Using the substitution (\ref{eq3.1}), one can rewrite
(\ref{eq3.3}) in the form:
\begin{equation}\label{eq3.4} 
W_L=t^2  ~W_{L^+}-W_{L^-}. 
\end{equation}

\bigskip
Thus the skein relation (\ref{eq2.6}) is equivalent to the relation 
(\ref{eq3.4}).  Lemma \ref{lm3.1}  is proved. 
\end{proof}

\begin{lemma}\label{lm3.2}
The exchange relations (\ref{eq1.1}) corresponding to the matrix $B$ 
given by formula (\ref{eq1.3}) imply an exchange relation:
\begin{equation}\label{eq3.5} 
T_{n+1}(x)=2(x_1x_4-x_2x_3)T_n(x)-T_{n-1}(x), 
\end{equation}
where $T_n(x)$ are the Chebyshev polynomials of the first kind.  
 \end{lemma}
\begin{proof}
In view of Theorem \ref{thm2.1},  the basis $\mathscr{B}$ of the 
cluster algebra $\mathcal{A}(\mathbf{x}, S_{0,2})$ has the form 
\begin{equation}\label{eq3.6} 
\mathscr{B}=\{x_i^px_{i+1}^q ~|~ p,q\ge 0\} ~\bigcup ~\{T_n(x_1x_4-x_2x_3) ~|~ n\ge 1\}. 
\end{equation}

\medskip
On the other hand, it is well known that the Chebyshev polynomials 
of the first kind satisfy the recurrence relation:
\begin{equation}\label{eq3.7} 
\begin{cases} 
T_0(x)=1&\cr
T_1(x)=x&\cr
T_{n+1}(x)=2x~T_n(x)-T_{n-1}(x).&
\end{cases}
\end{equation}

\medskip
Since the  $\mathcal{A}(\mathbf{x}, S_{0,2})$ is a cluster algebra of rank 2,
one can express via the exchange relation any cluster variable $x_i$ as a rational function of the
initial seed $\mathbf{x}=(x_1,x_2)$  [Williams 2014]  \cite{Wil1}.
We let $\mathbf{x}=(T_{n-1}(x), T_n(x))$, where $x=  x_1x_4-x_2x_3$. 
It follows from (\ref{eq3.7}) that the cluster variable  $T_{n+1}(x)$ is a rational 
function of the $T_{n-1}(x)$ and $T_n(x)$ of the form:
\begin{equation}\label{eq3.8} 
T_{n+1}(x)=2x~T_n(x)-T_{n-1}(x).
\end{equation}
Since such a rational function is unique, we conclude that (\ref{eq3.8}) 
is an exchange relation in the cluster algebra  $\mathcal{A}(\mathbf{x}, S_{0,2})$. 
Lemma \ref{lm3.2} is proved. 
\end{proof}
\begin{remark}\label{rmk3.3}
Apart from (\ref{eq3.8}) there are other exchange relations in  the cluster algebra $\mathcal{A}(\mathbf{x}, S_{0,2})$
coming from the  elements $x_i^px_{i+1}^q$ of the basis $\mathscr{B}$. 
It is not hard to see, that it is the reason why Pascal's triangle is Figure 2 is truncated 
[Jones 1991] \cite[p. 36, 50]{J1}
and one gets an inclusion  $\mathscr{J}\subset\mathcal{A}(\mathbf{x}, S_{0,2})$
rather than an isomorphism. 
\end{remark}
\begin{corollary}\label{cor3.4}
The skein relation (\ref{eq3.2}) is equivalent to the exchange relation 
(\ref{eq3.5}) modulo the equation
\begin{equation}\label{eq3.9} 
t^2=2 (x_1x_4-x_2x_3).
\end{equation}
\end{corollary}
\begin{proof}
The proof is comparison of equations (\ref{eq3.2}) and (\ref{eq3.5}),
where   
\begin{equation}\label{eq3.12+} 
\begin{cases} 
W_L(t)=T_{n+1}(x)&\cr
W_{L^+}(t)= T_n(x)&\cr
W_{L^-}(t)=T_{n-1}(x).&
\end{cases}
\end{equation}
\end{proof}

\begin{lemma}\label{lm3.5}
The equation (\ref{eq3.9}) is equivalent to the equation
\begin{equation}\label{eq3.10} 
t^2=2 \left({x_1^2+x_2^2+1\over x_1x_2}\right).
\end{equation}
\end{lemma}
\begin{proof}
Since the  $\mathcal{A}(\mathbf{x}, S_{0,2})$ is a cluster algebra of rank 2,
the cluster variables $x_3$ and $x_4$ in (\ref{eq3.9})  must be rational functions of 
the cluster variables $x_1$ and $x_2$. To find an explicit formula,  
we shall use  the exchange relations (\ref{eq2.7}). In our case $b=c=2$ and the exchange
relations (\ref{eq2.7})  take the form:
\begin{equation}\label{eq3.11} 
x_{i-1}x_{i+1}=x_i^2+1. 
\end{equation}

From (\ref{eq3.11}) one gets the following equations:
\begin{equation}\label{eq3.12} 
\begin{cases} 
x_3={x_2^2+1\over x_1}&\cr
x_4={x_3^2+1\over x_2}={x_1^2+(x_2^2+1)^2\over x_1^2x_2}.&
\end{cases}
\end{equation}

Using equations (\ref{eq3.12}) we conclude that:
\begin{equation}\label{eq3.13} 
x_1x_4-x_2x_3 = {x_1^2+x_2^2+1\over x_1x_2}. 
\end{equation}

Lemma \ref{lm3.5} follows from (\ref{eq3.9}) and (\ref{eq3.13}).
\end{proof}

\begin{lemma}\label{lm3.6}
 $\mathscr{J}\subset \mathcal{A}(\mathbf{x}, S_{0,2}).$
\end{lemma}
\begin{proof}
(i) Recall that  the ring $\mathscr{J}\subset  \mathbf{Z} [t^{\pm {1\over 2}}]$ is  generated by all
Jones polynomials $V_L(t)$.  The map $u\mapsto t^2$ defines an embedding  $\mathscr{J}\subset  \mathbf{Z} [t^{\pm 1}]$. 

\bigskip
(ii) On the other hand, each $V_L(u)$ can be obtained from the Jones polynomial $V_K(u)$ of the trivial link
using the skein relations (\ref{eq2.6}).  By  lemma \ref{lm3.1}, corollary \ref{cor3.4} and lemma \ref{lm3.5},
one gets an inclusion:
\begin{equation}\label{eq3.14}
\mathscr{J}\subset \mathcal{A}(\mathbf{x}, S_{0,2}),
\end{equation}
 where $\mathscr{J}$ is generated 
by the Chebyshev polynomials $T_n \left[2\left({x_1^2+x_2^2+1\over x_1x_2}\right)\right]$.  
We refer the reader to remark \ref{rmk3.3} for the extra details. 
Lemma \ref{lm3.6} is proved. 
\end{proof}

\begin{lemma}\label{lm3.7}
  $\mathscr{J}\cong K_0(\mathbb{P})$. 
\end{lemma}
\begin{proof}
(i)  Recall that the  cluster algebra $\mathcal{A}(\mathbf{x}, S_{0,2})$ has the structure of a 
dimension group \cite[Section 4.4]{Nik1}. Lemma \ref{lm3.6} says that $\mathscr{J}\subset \mathcal{A}(\mathbf{x}, S_{0,2})$
is a cluster sub-algebra, hence a dimension group with the order structure inherited from  the  $\mathcal{A}(\mathbf{x}, S_{0,2})$

\bigskip
(ii) On the other hand,  comparing the Bratteli diagrams in Figures 1 and 2,  we conclude that $\mathbb{P}\subset 
 \mathbb{A}(\mathbf{x}, S_{0,2})$ is an inclusion of the AF-algebras.   Thus one gets a commutative diagram shown 
 in Figure 3.

\bigskip
(iii) In view of item (i), we obtain  an isomorphism  $K_0(\mathbb{P})\cong \mathscr{J}$ from the diagram in Figure 3. 
Lemma \ref{lm3.7} is proved.
\end{proof}

\bigskip
Theorem \ref{thm1.1} follows from lemmas \ref{lm3.5}-\ref{lm3.7}.

\begin{figure}
\begin{picture}(300,110)(-60,-5)
\put(20,70){\vector(0,-1){35}}
\put(130,70){\vector(0,-1){35}}
\put(45,23){\vector(1,0){60}}
\put(45,83){\vector(1,0){60}}
\put(15,20){$\mathscr{J}$}
\put(115,20){$ \mathcal{A}(\mathbf{x}, S_{0,2})$}
\put(17,80){$\mathbb{P}$}
\put(115,80){ $ \mathbb{A}(\mathbf{x}, S_{0,2})$}
\put(50,30){\sf inclusion}
\put(50,90){\sf  inclusion}
\put(0,50){$K_0$}
\put(140,50){$K_0$}
\end{picture}
\caption{Functor $K_0$}
\end{figure}

\section{Examples}
We shall illustrate theorem \ref{thm1.1} by calculating the $V_L(t)$ 
of  two unlinked unknots, the Hopf link and the trefoil knot. 
We refer  the reader to the remarkable paper [Lee \& Schiffler 2019] \cite{LeeSch1} 
in which the cluster algebras are used to calculate the Jones polynomials 
of the 2-bridge knots. 

\begin{example}
(Two unknots)
Denote by  $S^1\cup S^1$ two unlinked unknots. To calculate the Jones 
polynomial $V_{S^1\cup S^1}(t)$ using theorem \ref{thm1.1}, we rewrite (\ref{eq3.8}) 
in the form:
\begin{equation}\label{eq4.1} 
T_{n-1}(x)=2x~T_n(x)-T_{n+1}(x).
\end{equation}
We let $n=1$ and 
\begin{equation}\label{eq4.2} 
\begin{cases} 
T_0(x)=W_{S^1\cup S^1}(t)=(t^2-1)V_{S^1\cup S^1}(t)&\cr
T_1(x)=W_{(S^1\cup S^1)^+}(t)=\left(-{t+1\over\sqrt{t}}\right)V_{(S^1\cup S^1)^+}(t)&\cr
T_2(x)=W_{(S^1\cup S^1)^-}(t)=\left(-{t+1\over\sqrt{t}}\right)V_{(S^1\cup S^1)^-}(t),&
\end{cases}
\end{equation}
compare with (\ref{eq3.1}).  From (\ref{eq3.2}) one gets:
\begin{equation}\label{eq4.3} 
(t^2-1)V_{S^1\cup S^1}(t) =t^2\left(-{t+1\over\sqrt{t}}\right)V_{(S^1\cup S^1)^+}(t)
-\left(-{t+1\over\sqrt{t}}\right)V_{(S^1\cup S^1)^-}(t).
\end{equation}

\medskip
Since $(S^1\cup S^1)^+\cong  (S^1\cup S^1)^-\cong K$ is the unknot, we have
\begin{equation}\label{eq4.4} 
V_{(S^1\cup S^1)^+}(t)=V_{(S^1\cup S^1)^-}(t)=1.
\end{equation}

\medskip
Using (\ref{eq4.4}) we calculate from equation (\ref{eq4.3}):
\begin{equation}\label{eq4.5} 
(t^2-1)V_{S^1\cup S^1}(t) =t^2\left(-{t+1\over\sqrt{t}}\right)
-\left(-{t+1\over\sqrt{t}}\right)=(t^2-1)\left(-{t+1\over\sqrt{t}}\right).
\end{equation}

\medskip
Thus from (\ref{eq4.5}) one gets the Jones polynomial of two unlinked unknots:
\begin{equation}\label{eq4.6} 
V_{S^1\cup S^1}(t) = \left(-{t+1\over\sqrt{t}}\right)= -t^{-{1\over 2}}-t^{1\over 2}.
\end{equation}
\end{example}

\begin{example}
(Hopf link)
Denote by  $H$ the Hopf link.  We let $n=2$ 
and the substitution (\ref{eq4.2}) brings   (\ref{eq4.1}) to the form:
\begin{equation}\label{eq4.7} 
(t^2-1)V_{K}(t) =t^2\left(-{t+1\over\sqrt{t}}\right)V_{S^1\cup S^1}(t)
-\left(-{t+1\over\sqrt{t}}\right)V_{H}(t).
\end{equation}

The $V_K(t)=1$ for the unknot $K$ and $V_{S^1\cup S^1}(t)=  -t^{-{1\over 2}}-t^{1\over 2}$
for the unlinked unknots $S^1\cup S^1$, see formula (\ref{eq4.6}).
The substitution of this data and subsequent reduction of  equation (\ref{eq4.7})
gives us the following equation:
\begin{equation}\label{eq4.8} 
\left({t+1\over\sqrt{t}}\right)V_{H}(t)=-t^3-t^2-t-1=(t+1)(-t^2-1).
\end{equation}

One gets easily from (\ref{eq4.8}) the Jones polynomial of the Hopf link:
\begin{equation}\label{eq4.9} 
V_{H}(t)=-t^{5\over 2}-t^{1\over 2}. 
\end{equation}
\end{example}

\begin{example}
(Trefoil knot)
Denote by  $T$ the trefoil knot.  We let $n=3$ 
and the substitution (\ref{eq4.2}) brings   (\ref{eq4.1}) to the form:
\begin{equation}\label{eq4.10} 
(t^2-1)V_{H}(t) =t^2\left(-{t+1\over\sqrt{t}}\right)V_{K}(t)
-\left(-{t+1\over\sqrt{t}}\right)V_{T}(t).
\end{equation}

The $V_K(t)=1$ for the unknot $K$ and $V_{H}(t)=-t^{5\over 2}-t^{1\over 2}$
for the Hopf link $H$, see formula (\ref{eq4.9}). The substitution of this data and 
subsequent reduction of  equation (\ref{eq4.10})
gives us the following equation:
\begin{equation}\label{eq4.11} 
\left({t+1\over\sqrt{t}}\right)V_{T}(t)=-t^{{9\over 2}} +t^{{5\over 2}} +t^{{3\over 2}}+t^{{1\over 2}}
=\left({t+1\over\sqrt{t}}\right) \left(-t^4+t^3+t\right).
\end{equation}

One gets  from (\ref{eq4.11}) the Jones polynomial of the trefoil knot:
\begin{equation}\label{eq4.12} 
V_{T}(t)=-t^4+t^3+t. 
\end{equation}
\end{example}

\section{Remarks}
An analog of theorem \ref{thm1.1} for the HOMFLY polynomials  
[Freyd, Yetter, Hoste, Lickorish, Millet \& Ocneanu 1985] 
\cite{FYHLMO} is proved  in \cite[Section 4.4.6.2]{Nik1}. 
In this case  $g=n=1$ and  $S_{1,1}$ is a torus with a cusp. The matrix $B$ associated 
to an ideal triangulation of the Riemann surface $S_{1,1}$    has the form
[Fomin,  Shapiro  \& Thurston  2008, Example 4.6]  \cite{FoShaThu1}: 
 \begin{equation}\label{eq5.1}
 B=\left(
 \begin{matrix}0 & 2 & -2\cr
              -2 & 0 & 2\cr
               2 & -2 & 0\end{matrix}
              \right).
 \end{equation}
It follows from  the exchange relations (\ref{eq1.1})  that the 
cluster $C^*$-algebra ${\Bbb A}(\mathbf{x}, S_{1,1})$ has  the 
Bratteli diagram shown in Figure 4.

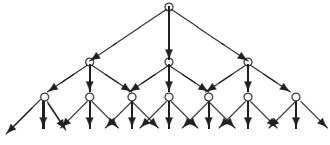
\begin{figure}
\begin{picture}(100,100)(0,130)

\put(50,200){\circle{3}}
\put(50,179){\circle{3}}
\put(20,179){\circle{3}}
\put(80,179){\circle{3}}


\put(3,166){\circle{3}}
\put(20,166){\circle{3}}
\put(36,166){\circle{3}}
\put(50,166){\circle{3}}
\put(65,166){\circle{3}}

\put(80,166){\circle{3}}
\put(98,166){\circle{3}}


\put(49,199){\vector(-3,-2){29}}
\put(51,199){\vector(3,-2){29}}
\put(50,200){\vector(0,-1){20}}


\put(19,178){\vector(-3,-2){15}}
\put(21,178){\vector(3,-2){15}}
\put(20,178){\vector(0,-1){10}}


\put(50,178){\vector(-3,-2){15}}
\put(50,178){\vector(3,-2){15}}
\put(50,178){\vector(0,-1){10}}


\put(79,178){\vector(-3,-2){15}}
\put(81,178){\vector(3,-2){15}}
\put(80,178){\vector(0,-1){10}}


\put(1.5,165){\vector(-1,-1){13}}
\put(4,164){\vector(2,-3){7}}
\put(2.5,164){\vector(0,-1){10}}

\put(19,164){\vector(-1,-1){10}}
\put(21,164){\vector(1,-1){10}}
\put(20,164){\vector(0,-1){10}}

\put(49,164){\vector(-1,-1){10}}
\put(51,164){\vector(1,-1){10}}
\put(50,164){\vector(0,-1){10}}

\put(79,164){\vector(-1,-1){10}}
\put(81,164){\vector(1,-1){10}}
\put(80,164){\vector(0,-1){10}}

\put(98,164){\vector(-1,-1){10}}
\put(100,164){\vector(1,-1){10}}
\put(98,164){\vector(0,-1){10}}


\put(65,164){\vector(-1,-1){10}}
\put(65,164){\vector(1,-1){10}}
\put(65,164){\vector(0,-1){10}}

\put(36,164){\vector(-1,-1){10}}
\put(36,164){\vector(1,-1){10}}
\put(36,164){\vector(0,-1){10}}


\end{picture}
\caption{Bratteli diagram of the algebra ${\Bbb A}(\mathbf{x}, S_{1,1})$.} 
\end{figure}

\begin{remark}
It is well known that the Jones polynomials can be obtained as a specialization
of the HOMFLY polynomials. 
This fact follows from an observation that the graph in Figure 1 is a sub-graph of
the graph in Figure 4. Hence  we get an inclusion of the cluster $C^*$-algebras: 
\begin{equation}\label{eq5.2}
\mathbb{A}(\mathbf{x}, S_{0,2})\subset\mathbb{A}(\mathbf{x}, S_{1,1}).
\end{equation}
Likewise, if one takes a double cover of the sphere by the torus 
ramified at  four points  taken for vertices of an ideal triangulation,
then the triangulation of $S_{0,2}$ is a sub-triangulation of the triangulation of $S_{1.1}$
[Fomin,  Shapiro  \& Thurston  2008]  \cite{FoShaThu1}.
Equivalently, the matrix (\ref{eq1.3}) can be obtained 
from a cancellation of the last row and column in the matrix  (\ref{eq5.1}).  
Unlike the case of the cluster algebras of rank 2 [Sherman \& Zelevinsky 2004]   \cite[Section 2]{SheZe1}, the canonical 
 bases  for the cluster algebras of rank 3 are unknown.  Therefore there is no immediate 
connection between the Chebyshev and HOMFLY polynomials. This fact can be viewed 
as a justification of the separate analysis of the Jones case. 
\end{remark}
\begin{remark}
A general result for the multivariable Laurent polynomials has been proved in \cite[Theorem 4.4.1]{Nik1}.
However, an explicit construction of such polynomials based on the exchange relations (\ref{eq1.1}) 
is unclear at the moment. 
\end{remark}

\bigskip\noindent
{\sf Acknowledgment.}  
We are grateful  to the referee for helpful comments.  

\bibliographystyle{amsplain}


\end{document}